\documentclass[11pt]{amsart}
\usepackage{amssymb, amscd}
\usepackage[all]{xy}

\theoremstyle{plain}
\newtheorem{thm}{Theorem}[section]
\newtheorem{prop}[thm]{Proposition}

\newtheorem{cor}[thm]{Corollary}

\theoremstyle{remark}

\numberwithin{equation}{section}

\newcommand{\Null}{\textnormal{Null}}
\newcommand{\del}{\partial}
\newcommand{\de}{\partial}
\newcommand{\dbar}{\overline{\del}}
\newcommand{\ddb}{\sqrt{-1}\del\dbar}
\newcommand{\ddbar}{\sqrt{-1}\del\dbar}

\newcommand{\ve}{\varepsilon}
\newcommand{\vp}{\varphi}
\newcommand{\ti}[1]{\tilde{#1}}

\newcommand{\Ric}{\mathrm{Ric}}
\renewcommand{\leq}{\leqslant}
\renewcommand{\geq}{\geqslant}

\title{Nakamaye's theorem on complex manifolds}
\author{Valentino Tosatti}
 \address{Department of Mathematics, Northwestern University, 2033 Sheridan Road, Evanston, IL 60208}
\email{tosatti@math.northwestern.edu}
\thanks{Supported in part by a Sloan Research Fellowship and NSF grant DMS-1308988.}

\begin{document}
\begin{abstract}
We discuss Nakamaye's Theorem and its recent extension to compact complex manifolds, together with some applications.
\end{abstract}
\maketitle
\section{Introduction}

Nakamaye's Theorem \cite{Na} characterizes the augmented base locus of a nef and big line bundle on a smooth projective variety over $\mathbb{C}$ as the union of all subvarieties where the restriction of the line bundle is not big (the so-called null locus of the bundle). While this result was originally motivated by applications to diophantine approximation problems \cite{Na3}, it has become an extremely useful tool in a variety of settings. Many generalizations of this result are known, including the cases of singular varieties, of $\mathbb{R}$-divisors, and of other ground fields \cite{Bi,CL, CMM, ELMNP,Na2}. This article discusses a transcendental generalization of Nakamaye's Theorem, which deals with $(1,1)$ classes on compact complex manifolds, and was obtained recently by Collins and the author \cite{CT}. The techniques we used are analytic, and thus also give a new proof of Nakamaye's original statement, and its generalization to $\mathbb{R}$-divisors.

We start in section \ref{sectbas} by giving a brief introduction to transcendental techniques in algebraic/analytic geometry, focusing on the basic results and definitions that we will need later, and explaining how these relate to their algebraic counterparts. In particular, we explain a result of Boucksom \cite{BoT} which gives a transcendental characterization of the augmented base locus of an $\mathbb{R}$-divisor. Nakamaye's Theorem and its generalization to complex manifolds are stated in section \ref{sectmain}. Section \ref{sectappl} contains several applications of this result, including new proofs (and generalizations) of theorems by Demailly-P\u{a}un \cite{DP}, Fujita-Zariski \cite{Fu,Za}, Takayama \cite{Ta}, a discussion of Seshadri constants for $(1,1)$ classes, and applications to the K\"ahler-Ricci flow and to Ricci-flat Calabi-Yau metrics. In section \ref{sectproof} we give a detailed outline of the proof of the main theorem, and we end with a brief discussion of the problem of effectivity in Nakamaye's Theorem in section \ref{secteff}.

While the techniques we use are analytic, and most of the results are stated for general compact complex manifolds, we have strived to make this article accessible to algebraic geometers, by providing the algebraic counterparts (or special cases) of the analytic results which are proved. It is our hope that this will make these ideas known to a wider audience.\\

{\bf Acknowledgments. }The author is grateful to the organizers of the 2015 AMS Summer Institute in Algebraic Geometry for the kind invitation to speak and to contribute this article to the proceedings. The author would also like to thank S. Boucksom, T. Collins, J.-P. Demailly, Y. Deng, L. Ein, R. Lazarsfeld, J. Lesieutre, N. McCleerey, M. Musta\c{t}\u{a}, M. P\u{a}un, M. Popa for very helpful discussions, as well as the referees for useful suggestions. Part of this paper was written during the author's visit at Harvard University's Center for Mathematical Sciences and Applications, which he would like to thank for the gracious hospitality.

\section{Basic results}\label{sectbas}
Throughout this paper, $X$ will denote a compact complex manifold (unless otherwise stated), of complex dimension $n$. We will denote by $\omega$ the real $(1,1)$ form associated to a Hermitian metric on $X$.

\subsection{(1,1) classes}
Let $\alpha$ be a closed real $(1,1)$ form on $X$, and denote by $[\alpha]$ its class in the (finite-dimensional) Bott-Chern cohomology group
$$H^{1,1}(X,\mathbb{R})=\frac{\{d\textrm{-closed real }(1,1) \textrm{ forms}\}}{\{\ddbar f\ |\ f\in C^\infty(X,\mathbb{R})\}}.$$
We will call $[\alpha]$ simply a $(1,1)$ class. If $L\to X$ is a holomorphic line bundle, and $h$ is a smooth Hermitian metric on $L$, then its curvature form is given locally by
$$R_h=-\frac{\sqrt{-1}}{2\pi}\partial\overline{\partial}\log h.$$
This defines a global closed real $(1,1)$ form on $X$, and if $h'$ is another metric on $L$ then the ratio $\frac{h}{h'}$ is a globally defined smooth positive function, and we have
$$R_h-R_{h'}=-\frac{\sqrt{-1}}{2\pi}\partial\overline{\partial}\log\frac{h}{h'}.$$
This shows that there is a well-defined class $c_1(L):=[R_h]\in H^{1,1}(X,\mathbb{R})$. We will say that two holomorphic line bundles $L,L'$ on $X$ are numerically equivalent if $c_1(L)=c_1(L')$.

The real vector subspace of $H^{1,1}(X,\mathbb{R})$ spanned by all classes of the form $c_1(L)$ as $L$ varies among all holomorphic line bundles on $X$ defines the real N\'eron-Severi group $N^1(X,\mathbb{R})\subset H^{1,1}(X,\mathbb{R})$, which is in general a strictly smaller subspace, and $(1,1)$ classes which are outside of it are usually referred to as transcendental.

\subsection{Positivity notions} We now introduce several basic notions of positivity for $(1,1)$ classes, which generalize the corresponding notions for line bundles over projective manifolds.

Let $[\alpha]$ be a $(1,1)$ class on a compact complex manifold $X$, where $\alpha$ is a closed real $(1,1)$ form. Recall that $\omega$ denotes a fixed Hermitian form on $X$. We define the following positivity notions (which are easily seen to be independent of the choice of $\omega$):

\begin{itemize}
\item $[\alpha]$ is {\bf K\"ahler} if it contains a representative which is a K\"ahler form, i.e. if there is a smooth function $\vp$ such that $\alpha+\ddbar\vp\geq\ve\omega$ on $X$, for some $\ve>0$.
\item $[\alpha]$ is {\bf nef} if for every $\ve>0$ there is a smooth function $\vp_\ve$ such that
$\alpha+\ddbar\vp_\ve\geq-\ve\omega$ holds on $X$.
\item $[\alpha]$ is {\bf big} if it contains a K\"ahler current, i.e. if there exists a quasi-plurisubharmonic (quasi-psh) function $\vp:X\to\mathbb{R}\cup\{-\infty\}$ such that $\alpha+\ddbar\vp\geq\ve\omega$ holds weakly as currents on $X$, for some $\ve>0$.
\item $[\alpha]$ is {\bf pseudoeffective} if it contains a closed positive current, i.e. if there exists a quasi-psh function $\vp:X\to\mathbb{R}\cup\{-\infty\}$ such that $\alpha+\ddbar\vp\geq0$ holds weakly as currents.
\end{itemize}

Here, a quasi-psh function means that locally it is given by the sum of a plurisubharmonic function plus a smooth function.

Clearly every K\"ahler class is nef and big, and every big class is pseudoeffective. Also, using weak compactness of currents in a fixed class, it is easy to see that every nef class is pseudoeffective (and there are in general no other implications among these notions).

As shown by Demailly \cite[Proposition 4.2]{Dem92}, if $X$ is projective and if $[\alpha]=c_1(L)$ for a holomorphic line bundle $L$, then these notions are equivalent to their algebraic counterparts. More precisely:
\begin{itemize}
\item $c_1(L)$ is K\"ahler iff $L$ is ample (this is just the Kodaira embedding theorem)
\item $c_1(L)$ is nef iff $L$ is nef (i.e. $(L\cdot C)\geq 0$ for all curves $C\subset X$)
\item $c_1(L)$ is big iff $L$ is big (i.e. $h^0(X,L^m)\geq c m^n$ for some $c>0$ and all large $m$)
\item $c_1(L)$ is pseudoeffective iff $L$ is pseudoeffective (i.e. $c_1(L)$ lies in the closed cone in $N^1(X,\mathbb{R})$ generated by classes of effective $\mathbb{R}$-divisors)
\end{itemize}
In particular, all of these notions are numerical.
Furthermore, these equivalences extend immediately to the case when we replace $L$ by an $\mathbb{R}$-divisor $D$.

The following result of Demailly-P\u{a}un \cite[Theorem 2.12]{DP} will be crucial:
\begin{thm}[Demailly-P\u{a}un \cite{DP}]\label{dpmass}
Let $X$ be a compact complex manifold in Fujiki's class $\mathcal{C}$ and $[\alpha]$ a $(1,1)$ class which is nef and satisfies
$$\int_X\alpha^n>0.$$
Then $[\alpha]$ is big.
\end{thm}
Recall here that $X$ being in Fujiki's class $\mathcal{C}$ \cite{Fu2} means that there exists a modification $\mu:\ti{X}\to X$, obtained as a composition of blowups with smooth centers, such that $\ti{X}$ is a compact K\"ahler manifold.

When $X$ is projective and $[\alpha]=c_1(L)$, Theorem \ref{dpmass} is just a simple consequence of Riemann-Roch, see \cite[Theorem 2.2.16]{Laz}. In general, this result uses the ``mass concentration'' technique for Monge-Amp\`ere equations, which was pioneered by Demailly \cite{Dem2}. A simpler proof of Theorem \ref{dpmass} was recently obtained by Chiose \cite{Ch} (see also \cite{To} for an exposition of this and related topics).

We close this subsection with a remark about closed positive currents. If $T=\alpha+\ddbar\vp\geq 0$ is such a current on $X$, where $\vp$ is quasi-psh, and $\mu:\ti{X}\to X$ is a holomorphic map such whose image is not contained in the locus $\{\vp=-\infty\}$, then we can define a pullback current $\mu^*T=\mu^*\alpha+\ddbar(\mu^*\vp)$, where $\mu^*\vp=\vp\circ\mu$, and this will still be closed positive on $\ti{X}$. In particular, if $\iota:V\to X$ is the inclusion of a submanifold, which is not contained in $\{\vp=-\infty\}$, then we will write
$T|_V:=\iota^*T$.

On the other hand, if $f:X\to Y$ is any holomorphic map between compact complex manifolds, the pushforward current $f_*T$ (defined as usual by duality, using the pullback map $f^*$ on differential forms) is also a closed positive current on $Y$. If $f$ is a modification, then we have that $f_*f^*$ acts as the identity on closed positive currents. From this it follows easily that in this case if $T$ is a K\"ahler current on $X$ then $f_*T$ is a K\"ahler current on $Y$.

\subsection{Base loci}
Let $L$ be a holomorphic line bundle over a compact complex manifold $X$. The base locus of $L$ is defined as
$${\rm Bs}(L)=\bigcap_{s\in H^0(X,L)}\{s=0\}.$$
This is a closed analytic subvariety of $X$. The stable base locus of $L$ is the closed analytic subvariety defined by
$$\mathbb{B}(L)=\bigcap_{m\geq 1} {\rm Bs}(L^m).$$
By the local Noetherian property of analytic subsets, there exists $m\geq 1$ such that $\mathbb{B}(L)={\rm Bs}(L^m)$. In general these loci also carry the structure of complex analytic subspaces of $X$ (or subschemes if $X$ is projective), but we will not make use of it, and when considering analytic subvarieties we always disregard this extra structure.

We can also define the stable base locus $\mathbb{B}(D)$ for $D$ a $\mathbb{Q}$-divisor, as the intersection of the base loci of $mD$ over all $m\geq 1$ such that $mD$ is an integral divisor, and hence defines a line bundle.

In general the stable base locus is not a numerical invariant, in the sense that there exist line bundles with the same first Chern class but with different stable base loci, see e.g. \cite[Example 10.3.3]{Laz}. To get around this issue, Nakamaye \cite{Na} introduced the augmented base locus of a line bundle $L$ over a projective manifold, defined by
$$\mathbb{B}_+(L)=\bigcap_{\ve\in\mathbb{Q}_{>0}}\mathbb{B}(L-\ve A),$$
where $A$ is any fixed ample line bundle over $X$, and $L-\ve A$ is regarded as a $\mathbb{Q}$-divisor. This is clearly a closed analytic subvariety of $X$, and it is easy to see that it is independent of the choice of $A$, and that it is a numerical invariant. A systematic study of augmented base loci was initiated in \cite{ELMNP2}, which also contains the proofs of these assertions.

It is also easy to extend this definition to $\mathbb{R}$-divisors, by setting
$$\mathbb{B}_+(D)=\bigcap_{A}\mathbb{B}(D-A),$$
where $D$ is an $\mathbb{R}$-divisor and the intersection is over all ample $\mathbb{R}$-divisors $A$ such that $D-A$ is a $\mathbb{Q}$-divisor. This definition agrees with the previous one when $D$ is an integral divisor (see \cite{ELMNP2}).

Apart from being a numerical invariant of the $\mathbb{R}$-divisor $D$, the augmented base locus $\mathbb{B}_+(D)$ has several useful properties. For example, $\mathbb{B}_+(D)\neq X$ iff $D$ is big, and $\mathbb{B}_+(D)=\emptyset$ iff $D$ is ample (see again \cite{ELMNP2}). If $L$ is a line bundle, then the complement of $\mathbb{B}_+(L)$ is the largest Zariski open subset such that for all large and divisible $m$ the Kodaira map
$$X\backslash \mathbb{B}(L)\to \mathbb{P}H^0(X,L^m),$$
defined by sections in $H^0(X,L^m)$ is an isomorphism onto its image (see \cite[Theorem A]{BCL}).

\subsection{The non-K\"ahler locus} In the previous subsection we defined the augmented base locus of an $\mathbb{R}$-divisor on a projective manifold. Following Boucksom \cite{Bo2}, we now generalize this to an arbitrary $(1,1)$ class $[\alpha]$ on a compact complex manifold $X$, by defining the non-K\"ahler locus $E_{nK}(\alpha)$ of $[\alpha]$.

If $[\alpha]$ is not big then we simply set $E_{nK}(\alpha)=X$, while if $[\alpha]$ is big (i.e. it contains K\"ahler currents) we set
$$E_{nK}(\alpha)=\bigcap_{T\in[\alpha]}\mathrm{Sing}(T),$$
where the intersection ranges over all K\"ahler currents $T=\alpha+\ddbar\vp$ in the class $[\alpha]$, and we have defined $\mathrm{Sing}(T)$ to be the complement of the set of points $x\in X$ such that $\vp$ is smooth near $x$. Boucksom observed in \cite[Theorem 3.17]{Bo2} that in fact there exists a K\"ahler current $T$ in $[\alpha]$ with
\begin{equation}\label{equal}
E_{nK}(\alpha)=\mathrm{Sing}(T).
\end{equation}
Furthermore, Demailly's fundamental regularization theorem for currents \cite{Dem} implies that we may assume that $T=\alpha+\ddbar \vp$ has analytic singularities, which means that there exist a coherent ideal sheaf $\mathcal{I}\subset \mathcal{O}_X$ and $c\in\mathbb{R}_{>0}$, such that for every $x\in X$ there exist an open neighborhood $U$ of $x$, finitely many generators $\{f_j\}$ of $\mathcal{I}$ over $U$ and a continuous function $h$ on $U$ such that
$$\vp=c\log\left(\sum_j|f_j|^2\right)+h,$$
holds on $U$. In particular, for such a current $T$ we have that $\mathrm{Sing}(T)$ is a closed analytic subvariety of $X$ (which is the underlying set of the analytic subspace of $X$ defined by $\mathcal{I}$). Therefore, $E_{nK}(\alpha)$ is always a closed analytic subvariety. We record this result as a theorem:

\begin{thm}[Boucksom \cite{Bo2}]\label{bouck}
Let $X$ be a compact complex manifold and $[\alpha]$ a big $(1,1)$ class. Then there exists a K\"ahler current on $X$ in the class $[\alpha]$ with analytic singularities precisely along the analytic set $E_{nK}(\alpha)$.
\end{thm}

This implies that if $[\alpha]$ is a big $(1,1)$ class, given any point $x\not\in E_{nK}(\alpha)$ we can find a global K\"ahler current $T$ on $X$ in the class $[\alpha]$ which is in fact a smooth K\"ahler form in a neighborhood of $x$. In particular, we see that $E_{nK}(\alpha)=\emptyset$ iff $[\alpha]$ is a K\"ahler class.

The connection with the algebraic setting is then provided by the following result, essentially due to Boucksom \cite[Corollary 2.2.8]{BoT}. We reproduce here the proof given in \cite[Proposition 2.4]{CT}.
\begin{thm}[Boucksom \cite{BoT}]\label{bplus} Let $X$ be a projective manifold, and $D$ an $\mathbb{R}$-divisor on $X$. Then
$$\mathbb{B}_+(D)=E_{nK}(c_1(D)).$$
\end{thm}
\begin{proof}
As we recalled earlier, we have that $D$ is big iff $c_1(D)$ is big.
Therefore if $D$ is not big then we have $\mathbb{B}_+(D)=E_{nK}(c_1(D))=X$, and so we may assume that $D$ is big.

First we show that $E_{nK}(c_1(D))\subset \mathbb{B}_+(D)$. If $x\not\in \mathbb{B}_+(D)$ then by definition there exists an ample $\mathbb{R}$-divisor $A$ such that $D-A$ is a $\mathbb{Q}$-divisor and its stable base locus does not contain $x$. Therefore there is $m\geq 1$ such that $m(D-A)$ is the divisor of a line bundle $L$ and there is a section $s\in H^0(X,L)$ with $s(x)\neq 0$. Complete it to a basis $\{s=s_1,s_2,\dots,s_N\}$ of $H^0(X,L)$, and fix a smooth Hermitian metric $h$ on $L$ with curvature form $R_h.$
Then
$$T=R_h+\frac{\sqrt{-1}}{2\pi}\partial\overline{\partial}\log\sum_i |s_i|^2_h,$$
is a closed positive current in $c_1(L)$, with analytic singularities and which is smooth near $x$.
If $\omega$ is a K\"ahler form in the class $c_1(A)$, then $\frac{1}{m}T+\omega$ is then a K\"ahler current on $X$ in $c_1(D)$ which is smooth near $x$, and so $x\not\in E_{nK}(c_1(D))$.

To see the reverse inclusion, assume $x\not\in E_{nK}(c_1(D))$ so that by Theorem \ref{bouck} we can find a K\"ahler current $T$ in the class $c_1(D)$ with analytic singularities which is smooth in a coordinate patch $U$ containing $x$. We can find an ample line bundle $A$ and a small $\delta'>0$ such that $D-\delta'A$ is a $\mathbb{Q}$-divisor. If $\omega$ is a K\"ahler form in $c_1(A)$, then there exists $0<\delta<\delta'$ with $\delta'-\delta\in\mathbb{Q}$ and such that $T-\delta\omega$ is still a K\"ahler current. Then $D-\delta A$ is a big $\mathbb{Q}$-divisor.
Let $\theta$ be a smooth cutoff function supported in $U$ and identically $1$ near $x$, and let
$$\ti{T}=T-\delta\omega+\ve\ddb (\theta\log|z-x|^2),$$
where $\ve>0$ is chosen small enough so that $\ti{T}$ is a K\"ahler current in $c_1(D-\delta A)$. By construction, the Lelong number $\nu(\ti{T},x)$ is equal to $\ve$, and $\ti{T}$ is smooth on $U\backslash\{x\}$. Let $\gamma$ be a smooth representative of $c_1(K_X)$, and let
$\ti{T}_m=m\ti{T}-\gamma$. For $m$ sufficiently large, $\ti{T}_m$ is a K\"ahler current in $c_1(m(D-\delta A)-K_X)$, which is smooth on $U\backslash\{x\}$ and with Lelong number $\nu(\ti{T}_m,x)=m\ve$.
If $m$ is sufficiently divisible, then $m(D-\delta A)-K_X$ is the divisor of a holomorphic line bundle $L_m$, which is big. Applying \cite[Corollary 3.3]{Dem92} (which is an application of H\"ormander's $L^2$ estimates for $\overline{\partial}$) we see that if we choose $m$ large so that $m\ve\geq 1$, then $H^0(X,K_X+L_m)=H^0(X,m(D-\delta A))$ generates $0$-jets at $x$.
Hence, there is a global section of $m(D-\delta A)$ which does not vanish at $x$, and so $x\not\in \mathbb{B}_+(D)$.
\end{proof}

The following result is well-known (see e.g. \cite[Lemma 3.1]{CT}, and \cite[Proposition 1.1]{ELMNP} in the algebraic case).
\begin{prop}\label{isola}
Let $X$ be a compact complex manifold and $[\alpha]$ a real $(1,1)$ class. Then $E_{nK}(\alpha)$ does not have any isolated points.
\end{prop}

In particular, this property holds for the augmented base locus $\mathbb{B}_+(L)$ for any line bundle $L$ on a projective manifold. In fact, it also holds for the stable base locus $\mathbb{B}(L)$, but this is a much deeper result of Zariski \cite{Za}, which we will present below in subsection \ref{zars}.

\begin{proof}
Recall that $E_{nK}(\alpha)=X$ iff $[\alpha]$ is not big. Therefore we may assume that $[\alpha]$ is big.
Assume that $x$ is an isolated point in $E_{nK}(\alpha)$, and choose a chart $U$ centered at $x$ with coordinates $\{z_1,\dots,z_n\}$ such that $U\cap E_{nK}(\alpha)=\{x\}.$ Choose a constant $A>0$ sufficiently large so that
$$\alpha+A \ddbar|z|^2\geq \omega,$$
on $U$. Choose $K=\alpha+\ddbar\vp$ a K\"ahler current in the class $[\alpha]$ with analytic singularities along $E_{nK}(\alpha)$, which exists by Theorem \ref{bouck}. In particular, $\vp$ is smooth on $U\backslash\{x\}$ and $\vp(x)=-\infty$. If $\widetilde{\max}$ denotes a regularized maximum function (see \cite[I.5.18]{Demb}), then we can choose a large constant $C>0$ so that the function on $U$
$$\psi:=\widetilde{\max}(\vp,A|z|^2-C),$$
is equal to $\vp$ near $\de U$ and equal to $A|z|^2-C$ near $x$. Hence $\psi$ is smooth on $U$, and it glues to $\vp$ on $X\backslash U$ to give a global function $\psi$ with analytic singularities so that $\alpha+\ddbar\psi$ is a K\"ahler current on $X$ in the class $[\alpha]$ with analytic singularities and smooth near $x$. Hence $x\not\in E_{nK}(\alpha)$, which is a contradiction.
\end{proof}

For later applications, we need one more property of the non-K\"ahler locus, which was observed in \cite[Proposition 2.3]{BBP} in the algebraic setting.
\begin{prop}\label{pull}
Let $\mu:\ti{X}\to X$ be a modification between compact complex manifolds. If $[\alpha]$ is a $(1,1)$ class on $X$ then
$$E_{nK}(\mu^*\alpha)=\mu^{-1}(E_{nK}(\alpha))\cup\mathrm{Exc}(\mu),$$
where $\mathrm{Exc}(\mu)$ is the exceptional locus of $\mu$.
\end{prop}
\begin{proof}
Since the pushforward of a K\"ahler current is also a K\"ahler current, we see that if $[\alpha]$ is not big then $\mu^*[\alpha]$ is not big either.
As remarked earlier, $[\alpha]$ is not big iff $E_{nK}(\alpha)=X$, and therefore we may assume that $[\alpha]$ is big.
We may also assume that $\mathrm{Exc}(\mu)\neq \emptyset$, otherwise
$\mu$ is a biholomorphism and the result is obvious.

First, let $x\not\in E_{nK}(\mu^*\alpha)$ so there exists a K\"ahler current $T=\mu^*\alpha+\ddbar\vp$ which is smooth near $x$.  If $x\in \mathrm{Exc}(\mu),$ Zariski's main theorem
implies that there exists an irreducible component $E$ of the fiber $\mu^{-1}(\mu(x))$ which is positive dimensional, although it may not be smooth.
 Then $T$ can be restricted to $E$ since $T$ is smooth near $x$, and we have
$$T|_E=(\mu^*\alpha+\ddbar\vp)|_E=\ddbar(\vp|_E)\geq \ve\ti{\omega}|_E,$$
for some $\ve>0$ and Hermitian form $\ti{\omega}$ on $\ti{X}$. Hence $\vp|_E$ is strictly plurisubharmonic on the analytic space $E$, but since $E$ is compact and connected $\vp|_E$ must be constant, which is a contradiction.

Therefore $x\not\in \mathrm{Exc}(\mu)$, and so $\mu$ is an isomorphism near $x$. Hence $\mu_*T$ is a K\"ahler current on $X$ in the class $[\alpha]$ which is smooth near $\mu(x)$, i.e.
$\mu(x)\not\in E_{nK}(\alpha)$. We have thus proved that $\mu^{-1}(E_{nK}(\alpha))\cup\mathrm{Exc}(\mu)\subset E_{nK}(\mu^*\alpha)$.

Assume conversely that $x\not\in \mu^{-1}(E_{nK}(\alpha))\cup\mathrm{Exc}(\mu)$. First, let
us assume that $\mu$ is a composition of blowups with smooth centers. Then $\mu$ is an isomorphism near $x$ and there is a K\"ahler current $T$ on $X$ in the class $[\alpha]$ which is smooth near $\mu(x)$. Then $\mu^*T$ is a closed positive current on $\ti{X}$ in the class $\mu^*[\alpha]$ which is smooth near $x$, and which satisfies
$\mu^*T\geq\ve\mu^*\omega$ as currents on $\ti{X}$, where $\omega$ is a Hermitian form on $X$ and $\ve>0$.

We now perturb $\mu^*T$ in its class to make it a K\"ahler current.
More specifically, we claim that there is a quasi-psh function $f$, smooth near $x$, such that $\mu^*\omega+\ddbar f$ is a K\"ahler current on $\ti{X}$.
Consider first the case when $\mu$ is the blowup of $X$ along a smooth submanifold, with exceptional divisor $E$. Then it is well-known (see e.g. \cite[Lemma 3.5]{DP}) that we can find $\delta>0$ and a smooth closed real $(1,1)$ form $\eta$ on $\ti{X}$, in the same cohomology class as the current of integration $[E]$, such that
$\mu^*\omega-\delta\eta$ is a K\"ahler form on $\ti{X}$. Writing $\eta=[E]-\ddbar h$ for some quasi-psh function $h$, smooth off $E$, we have that
\begin{equation}\label{kc}
\mu^*\omega+\delta\ddbar h=\mu^*\omega-\delta\eta+\delta[E]\geq \mu^*\omega-\delta\eta,
\end{equation}
is a K\"ahler current. So in this case the claim is proved, with $f=\delta h$. The general case when $\mu$ is a composition of blowups with smooth centers follows similarly.

It follows that $\mu^*T+\ve\ddbar f\geq \ve\mu^*\omega+\ve\ddbar f$ is a K\"ahler current on $\ti{X}$, in the class $\mu^*[\alpha]$, which is smooth near $x$. Therefore, $x\not\in E_{nK}(\mu^*\alpha)$, as desired.

Lastly, we consider the case when $\mu$ is a general modification. By resolving the indeterminacies of the bimeromorphic map $\mu^{-1}:X\dashrightarrow \ti{X}$, we obtain a modification
$\nu:Y\to X$, which is a composition of blowups with smooth centers, and a holomorphic map $\mu':Y\to\ti{X}$ such that $\mu\circ\mu'=\nu$, and $\nu$ is an isomorphism outside $\mu(\mathrm{Exc}(\mu)),$ so
in particular near $\mu(x)$.
The map $\mu'$ is bimeromorphic, and it is an isomorphism near $x$ because $x\not\in \mathrm{Exc}(\mu)$. We also have $\mu(x)\not\in E_{nK}(\alpha)$ by assumption.
If $y\in Y$ is the preimage of $\mu(x)$ under the local isomorphism $\nu$, then there is a K\"ahler current $T$ on $X$ in the class $[\alpha]$ which is smooth near $\mu(x)$, and then
$\nu^*T$ is a closed positive current on $Y$ in the class $\nu^*[\alpha]$ which is smooth near $y$ and satisfies
$\nu^*T\geq\ve\nu^*\omega$ as currents on $Y$, where $\omega$ is a Hermitian form on $X$ and $\ve>0$. Since $\nu$ is a composition of blowups with smooth centers, our previous argument shows that
there is a K\"ahler current $\ti{T}$ on $Y$ in the class $\nu^*[\alpha]$ which is smooth near $y$. Then $\mu'_{*}\ti{T}$ is a K\"ahler current on $\ti{X}$
in the class $\mu'_{*}\nu^*[\alpha]=\mu'_{*}\mu'^*\mu^*[\alpha]=\mu^*[\alpha],$ smooth near $x$, i.e. $x\not\in E_{nK}(\mu^*\alpha)$, as required.
\end{proof}

\section{The main theorem}\label{sectmain}

\subsection{The null locus}
Let $X$ be a projective manifold and $L$ a nef line bundle over $X$. As mentioned earlier, $c_1(L)$ is nef, in the sense that for every $\ve>0$ there exists a smooth function $\vp_\ve$ such that $R_h+\ddbar\vp_\ve\geq -\ve\omega$ on $X$, where $\omega$ is a K\"ahler form on $X$, $h$ is any fixed smooth Hermitian metric on $L$ and $R_h$ is its curvature form (which represents $c_1(L)$). In particular, if $V\subset X$ is a positive-dimensional irreducible analytic subvariety of $X$ then we have
\begin{equation}\label{integg}\begin{split}
(V\cdot L^{\dim V})&=\int_V c_1(L)^{\dim V}=\int_V (R_h+\ddbar\vp_\ve)^{\dim V}\\
&=\lim_{\ve\to 0}\int_V (R_h+\ddbar\vp_\ve+\ve\omega)^{\dim V}\geq 0,
\end{split}\end{equation}
which is a result of Kleiman \cite{Kl} (which however was used in \cite{Dem92} to show that $c_1(L)$ is nef). Here the integrals of forms over $V$ are really improper integrals over the regular part of $V$, and the fact that these are finite, together with the justification of Stokes' Theorem, is a classical result of Lelong (see e.g. \cite{GH}).

Following Nakamaye \cite{Na} and Keel \cite{Ke} we define the null locus of $L$ to be
$$\Null(L)=\bigcup_{(V\cdot L^{\dim V})=0}V,$$
where the union is over all irreducible positive-dimensional analytic subvarieties $V\subset X$ with $(V\cdot L^{\dim V})=0$. This is clearly a numerical invariant of $L$, and it is in fact a closed analytic subvariety of $X$, although this is not entirely obvious. Clearly, if $L$ is ample then $\Null(L)=\emptyset$, and the converse is also true thanks to the Nakai-Moishezon ampleness criterion (see e.g. \cite[Theorem 1.2.23]{Laz}).

We have that $\Null(L)\neq X$ iff $(X\cdot L^n)>0$, and since $L$ is nef, this is true iff $L$ is big (using Riemann-Roch, see \cite[Theorem 2.2.16]{Laz}). More generally, $\Null(L)$ is just the union of all irreducible subvarieties $V$ such that $L|_V$ is not big.

As we have just seen, $\Null(L)$ satisfies the same basic properties as $\mathbb{B}_+(L)$, and both measure the failure of ampleness of $L$. Nakamaye's striking observation is that in fact these two loci coincide, and this is the content of:

\subsection{Nakamaye's Theorem}\label{nakas}
\begin{thm}[Nakamaye \cite{Na}]\label{naka}
Let $X$ be a projective manifold and $L$ a nef and big line bundle over $X$. Then
$$\mathbb{B}_+(L)=\Null(L).$$
\end{thm}
Of course this theorem is still true if $L$ is nef but not big, in which case $\mathbb{B}_+(L)=\Null(L)=X$.

Clearly Nakamaye's Theorem also holds when $L$ is replaced by a $\mathbb{Q}$-divisor $D$. In fact it also holds for nef $\mathbb{R}$-divisors, as shown by Ein-Lazarsfeld-Musta\c{t}\u{a}-Nakamaye-Popa \cite[Corollary 5.6]{ELMNP}.

There are many further generalizations of this theorem, including the case of positive characteristic \cite{CMM}, the case when $X$ is singular \cite{CL, Bi}.

We consider here its generalization to $(1,1)$ classes on complex manifolds. Let $[\alpha]$ be a nef $(1,1)$ class on a compact complex manifold $X$. If the manifold is K\"ahler, and if $V\subset X$ is any irreducible analytic subvariety then the same argument as in \eqref{integg} shows that
\begin{equation}\label{posi}
\int_V \alpha^{\dim V}\geq 0.
\end{equation}

If the manifold is not K\"ahler, it is not clear whether \eqref{posi} holds. But if we assume that $[\alpha]$ is nef and big, then
the existence of a big $(1,1)$ class implies that
$X$ is in class $\mathcal{C}$ by \cite[Theorem 3.4]{DP}.
Furthermore, every irreducible analytic subvariety $V\subset X$ is in class $\mathcal{C}$ as well \cite[Lemma 4.6]{Fu2}. Taking a K\"ahler modification of $V$ we conclude as above that \eqref{posi} holds.

Motivated by this, we define the null locus of a nef and big class $[\alpha]$ on a compact complex manifold to be
$$\Null(\alpha)=\bigcup_{\int_V\alpha^{\dim V}=0}V,$$
which is consistent with the algebraic definition if $[\alpha]=c_1(L)$. The main theorem of \cite{CT} is then the following:

\begin{thm}[Collins-T. \cite{CT}]\label{ct}
Let $X$ be a compact complex manifold and $[\alpha]$ a nef and big $(1,1)$ class on $X$. Then
\begin{equation}\label{key}
E_{nK}(\alpha)=\Null(\alpha).
\end{equation}
\end{thm}

Using Theorem \ref{bplus}, we immediately see that Theorem \ref{ct} implies Nakamaye's Theorem \ref{naka}, and this gives a new, analytic proof of that result (as well as its extension to $\mathbb{R}$-divisors \cite{ELMNP}).
\section{Applications}\label{sectappl}
In this section we give some applications of Theorem \ref{ct} to various related topics. Further applications of Theorem \ref{ct} to the Minimal Model Program for K\"ahler manifolds as well as to Zariski decompositions, we which did not include for the sake of conciseness, can be found in \cite{CHP,CH,Den,HP}.

\subsection{The Demailly-P\u{a}un Theorem}

\begin{thm}[Demailly-P\u{a}un \cite{DP}]\label{dp}
Let $X$ be a compact K\"ahler manifold, and let
$\mathcal{P}\subset H^{1,1}(X,\mathbb{R})$ be the cone of $(1,1)$ classes $[\alpha]$ which satisfy
$$\int_V\alpha^{\dim V}>0,$$
for all positive-dimensional irreducible analytic subvarieties $V\subset X$. Then the K\"ahler cone of $X$ is one of the connected components of $\mathcal{P}$.
\end{thm}
Furthermore, if $X$ is a projective manifold, then it is not hard to see that in fact the K\"ahler cone equals $\mathcal{P}$. This is a vast generalization of the Nakai-Moishezon ampleness criterion for line bundles (see e.g. \cite[Theorem 1.2.23]{Laz}).

\begin{proof}
Clearly every K\"ahler class is in $\mathcal{P}$, and the K\"ahler cone is open and convex, hence connected, and so it suffices to show that it is closed inside $\mathcal{P}$. It follows easily from the definition that $(1,1)$ classes in the closure of the K\"ahler cone are nef (and conversely). Assume then that $[\alpha]$ is a nef $(1,1)$ class which is also in $\mathcal{P}$. In particular, $$\int_X\alpha^n>0,$$
and so Theorem \ref{dpmass} shows that $[\alpha]$ is big, and by assumption $\Null(\alpha)=\emptyset$. By Theorem \ref{ct} we have that $E_{nK}(\alpha)=\emptyset$, i.e. $[\alpha]$ is a K\"ahler class, as required.
\end{proof}

Note that Theorem \ref{dp} (or Theorem \ref{ct}) also shows that if $[\alpha]$ is a nef class on a compact K\"ahler manifold then $\Null(\alpha)=\emptyset$ iff $[\alpha]$ is K\"ahler, a fact that we will use later.

An extension of Theorem \ref{dp} to the case when $X$ is a singular compact K\"ahler analytic space, embedded in a smooth ambient manifold, was recently obtained in \cite{CT3}, also as an application of Theorem \ref{ct}.

\subsection{The Fujita-Zariski Theorem}\label{zars} Following a strategy suggested by M. P\u{a}un (which was communicated to the author by R. Lazarsfeld), we give an analytic proof of a well-known theorem of Fujita \cite[Theorem 1.10]{Fu} (see also \cite{Ein} for another algebraic proof), and its extension to complex manifolds.

\begin{thm}\label{main}
Let $X$ be a compact complex manifold and $L$ a holomorphic line bundle. If the restriction of $L$ to its base locus ${\rm Bs}(L)$ is ample, then $L$ is semiample (i.e. $L^m$ is base point free for some $m\geq 1$).
\end{thm}

As an immediate corollary, we obtain a generalization of a classical theorem of Zariski \cite{Za}:
\begin{cor}\label{zar}
Let $X$ be a compact complex manifold and $L$ a holomorphic line bundle. If the base locus of $L$ is a finite set, then $L$ is semiample.
\end{cor}

Following the argument of \cite[Proposition 1.1]{ELMNP}, we also obtain (compare with Proposition \ref{isola}):
\begin{cor}
Let $X$ be a compact complex manifold and $D$ a $\mathbb{Q}$-divisor. Then the stable base locus $\mathbb{B}(D)$ does not have isolated points.
\end{cor}
\begin{proof}
Let $x$ be an isolated point of $\mathbb{B}(D)$, and choose $m\geq 1$ large so that $\mathbb{B}(D)={\rm Bs}(mD)$ and $mD$ is the divisor of a line bundle $L$.
Take a resolution $\mu:\ti{X}\to X$ of ${\rm Bs}(L)\backslash \{x\},$ as a complex analytic subspace of $X$, so that $\mu^*L=M+F$ with $F$ effective and with ${\rm Bs}(M)=\{\mu^{-1}(x)\}$.
By Corollary \ref{zar}, $M$ is semiample so $M^\ell$ is base point free for some $\ell\geq 1$, which implies that $x$ is not in the base locus of $\ell mD$, a contradiction.
\end{proof}

\begin{proof}[Proof of Theorem \ref{main}]
Since we are trying to show that $L$ is semiample, up to replacing $L$ with $L^m$ for $m$ large, we may assume that ${\rm Bs}(L)=\mathbb{B}(L)$. Our goal is to show by contradiction that $\mathbb{B}(L)=\emptyset.$ We can clearly assume that $\mathbb{B}(L)\neq X$.

Fix a smooth closed real $(1,1)$ form $\alpha$ on $X$ in the class $c_1(L)$, which is the curvature of a smooth Hermitian metric $h$ on $L$.
Write $\mathbb{B}(L)=\cup_{j=1}^N Z_j$ for the decomposition into irreducible components. For each $j$, since $L|_{Z_j}$ is ample, by \cite[Proposition 3.3 (i)]{DP} there exists an open
neighborhood $U_j$ of $Z_j$ in $X$ and a smooth real-valued function $\vp_j$ such that $\alpha+\ddbar\vp_j>0$ on $U_j$. Then the gluing lemma \cite[Lemme, p.419]{Pa} gives us a smooth function $\vp$ on $U=\cup_{j=1}^N U_j$ (a neighborhood of $\mathbb{B}(L)$ in $X$) with $\alpha+\ddbar\vp>0$ on $U$.

Choose $s_1,\dots,s_\ell\in H^0(X,L),$ such that $\{s_1=\dots=s_\ell=0\}=\mathbb{B}(L)$.
Let
$$\psi=\log\sum_{j=1}^\ell|s_j|^2_{h},$$
so that $\psi$ is smooth on $X\backslash \mathbb{B}(L)$, $\psi$ approaches $-\infty$ on $\mathbb{B}(L)$, and $\alpha+\ddbar\psi\geq 0$ as currents on $X$.
If $\widetilde{\max}$ denotes a regularized maximum (see \cite[I.5.18]{Demb}), then for $A>0$ large enough the function on $U$ given by
$$\widetilde{\max}(\vp-A,\psi),$$
coincides with $\psi$ in a neighborhood of $\de U$ and is smooth everywhere on $U$ (since it agrees with $\vp$ in a neighborhood of $\mathbb{B}(L)$). Therefore if we let $\Psi$ be equal to this function on $U$ and to $\psi$ on $X\backslash U$, then $\Psi$ is smooth on $X$, it satisfies $\alpha+\ddbar\Psi\geq 0$ everywhere and $\alpha+\ddbar\Psi>0$ in a neighborhood of $\mathbb{B}(L)$. In particular, $L$ is Hermitian semipositive (hence $c_1(L)$ is nef), and since $\alpha+\ddbar\Psi>0$ at at least one point of $X$, it follows from Siu \cite{Si2} (see also \cite{Dem87}) that $X$ is Moishezon (hence in class $\mathcal{C}$) and that $L$ is big.

Let then $x$ be any point in $\mathbb{B}(L)$, and $V$ any irreducible positive-dimensional subvariety of $X$ which passes through $x$. Then
\begin{equation}\label{pos}
\left(L^{\dim V}\cdot V\right)=\int_V (\alpha+\ddbar\Psi)^{\dim V}>0,
\end{equation}
and it follows that $x\not\in\mathrm{Null}(L)$
and by
Theorem \ref{ct}, $x\not\in E_{nK}(c_1(L))$.
Therefore there is a K\"ahler current $T$ in the class $c_1(L)$ with analytic singularities which is smooth near $x$. If $\mu:\ti{X}\to X$ is a resolution of the singularities
of $T$, obtained as a composition of blowups with smooth centers, then $\ti{X}$ is projective (see e.g. \cite[Remark 3.6]{DP}) and $\mu$ is an isomorphism near $x$. Thanks to Proposition \ref{pull}, we have $E_{nK}(c_1(\mu^*L))=\mu^{-1}(E_{nK}(c_1(L)))\cup\mathrm{Exc}(\mu),$ and so $\mu^{-1}(x)\not\in E_{nK}(c_1(\mu^*L))$. Since
$\ti{X}$ is projective, by Theorem \ref{bplus} we have
$$E_{nK}(c_1(\mu^*L))=\mathbb{B}_+(\mu^*L)\supset\mathbb{B}(\mu^*L),$$
and so $\mu^{-1}(x)\not\in\mathbb{B}(\mu^*L)$ and hence
$x\not \in \mathbb{B}(L),$ which is a contradiction.

Alternatively, we could have also used the K\"ahler current $T$ directly with the H\"ormander's $L^2$ estimates for $\overline{\partial}$ as in Theorem \ref{bplus} to construct a section of $L^m$ (for some $m\geq 1$) which does not vanish at $x$.
\end{proof}

\subsection{A local ampleness criterion}
The following is a transcendental generalization of a local ampleness criterion of Takayama \cite[Proposition 2.1]{Ta}, with the extra assumption that the class be nef:

\begin{thm}\label{localamp}
Let $X$ be a compact complex manifold and $[\alpha]$ a nef and big $(1,1)$ class. Let $T\geq 0$ be a closed positive current in the class $\alpha$ which is a smooth K\"ahler form on a nonempty open set $U\subset X$. Then we have that $U\cap E_{nK}(\alpha)=\emptyset$.
\end{thm}

Such a result (in the projective case) was used to obtain quasi-projectivity criteria in \cite[Theorem 6.1]{LWX} and \cite[Theorem 6]{ST}.

\begin{proof}
Given any $x\in U$ let $V$ be any irreducible subvariety of $X$ which passes through $V$, and say $\dim V=k$.
Let $\mu:\ti{X}\to X$ be an embedded resolution of singularities of $V\subset X$, so that $\mu$ is a composition of blowups with smooth centers, it is an isomorphism at the generic point of $V$, and the proper transform $\ti{V}$ of $V$ is smooth. The class $\mu^*[\alpha]$ is nef and big and satisfies
$$\int_V\alpha^k=\int_{\ti{V}}(\mu^*\alpha)^k,$$
and by a theorem of by Boucksom \cite[Theorem 4.1]{Bo} this equals the volume of the class $[\mu^*\alpha|_{\ti{V}}]$, which is defined as
$$\sup_{S\geq 0}\int_{\ti{V}} S_{ac}^{k},$$
where the supremum is over all closed positive currents $S$ in this class, and $S_{ac}$ denotes the absolutely continuous part of $S$ (see \cite{Bo} for details).
Observe that $\mu^*T$ is a closed positive current in the class $\mu^*[\alpha]$ which is a smooth semipositive $(1,1)$ form on the open set $\mu^{-1}(U)$, and in fact a K\"ahler form on
the open subset $\mu^{-1}(U)\backslash \mathrm{Exc}(\mu).$ Therefore $(\mu^*T)|_{\ti{V}}$ is a K\"ahler form on a nonempty open subset of $\ti{V}$ (where of course it equals its absolutely continuous part), and therefore
$$\int_V\alpha^k\geq \int_{\ti{V}} ((\mu^*T)|_{\ti{V}})_{ac}^k>0.$$
Since $V$ is arbitrary, we conclude that
$x\not\in\Null(\alpha)$, and so it follows from Theorem \ref{ct} that $x\not\in E_{nK}(\alpha)$, as required.
\end{proof}

\subsection{Seshadri constants}
These were introduced by Demailly \cite{Dem92} to measure the local positivity of line bundles. Further properties of these invariants can be found for example in \cite{Sesh, Dem2, Laz}.

Let $[\alpha]$ be a nef $(1,1)$ class on a compact K\"ahler manifold. We define its Seshadri constant at a point $x\in X$ to be
$$\ve(\alpha,x)=\sup\{ \lambda\geq 0\ |\ \mu^*[\alpha]-\lambda [E]\ {\rm is\ nef}\},$$
where $\mu:\ti{X}\to X$ is the blowup of $X$ at $x$, and $E=\mu^{-1}(x)$ is the exceptional divisor. The function $\ve(\cdot, x)$ is continuous on the nef cone.

It is natural to ask what are the points $x\in X$ where $\ve(\alpha,x)$ vanishes. This is answered by the following result, which also contains
the extension of \cite[Proposition 5.1.9]{Laz} in our transcendental situation:
\begin{thm}\label{subvar}
Let $X$ be a compact K\"ahler manifold, $[\alpha]$ a nef $(1,1)$ class and $x\in X$. Then we have
\begin{equation}\label{zero}
\ve(\alpha,x)=0\quad\textrm{if and only if }\quad x\in E_{nK}(\alpha)=\Null(\alpha),
\end{equation}
and furthermore for all $x\in X$ we have
\begin{equation}\label{uno}
\ve(\alpha,x)=\inf_{V\ni x}\left(\frac{\int_{V}\alpha^{\dim V}}{\mathrm{mult}_x V}\right)^{\frac{1}{\dim V}},
\end{equation}
where the infimum runs over all positive-dimensional irreducible analytic subvarieties $V$ containing $x$, and $\mathrm{mult}_x V$ denotes the multiplicity of $V$ at $x$.
The infimum is in fact achieved if $[\alpha]$ is K\"ahler.
\end{thm}
\begin{proof}
First we show that for all $x\in X$ we have
\begin{equation}\label{part}
\ve(\alpha,x)\leq \inf_{V\ni x}\left(\frac{\int_{V}\alpha^{\dim V}}{\mathrm{mult}_x V}\right)^{\frac{1}{\dim V}}.
\end{equation}
To this end, let $V$ be any positive-dimensional irreducible subvariety through $x$, let $\mu:\ti{X}\to X$ be the blowup of $x$, let $\ti{V}$ be the proper transform of $V$ through $\mu$, and let $k=\dim V>0$. If $\lambda\geq 0$ is such that $\mu^*[\alpha]-\lambda[E]$ is nef then
$$0\leq \int_{\ti{V}}(\mu^*\alpha-\lambda[E])^{k}=\int_V\alpha^k+(-1)^k\lambda^k\int_{\ti{V}}[E]^k=\int_V\alpha^k-\lambda^k\mathrm{mult}_x V,$$
using the fact that $\mathrm{mult}_x V=-(-1)^k\int_{\ti{V}}[E]^k$ thanks to \cite[Lemma 5.1.10]{Laz}.
This proves \eqref{part}, which also clearly shows that if $x\in \Null(\alpha)$ then $\ve(\alpha,x)=0$.

Conversely if $x\not\in \Null(\alpha)=E_{nK}(\alpha)$ (using Theorem \ref{ct}), then there is a K\"ahler current $T\in[\alpha]$ with analytic singularities which is smooth near $x$, and satisfies $T=\alpha+\ddbar\psi\geq\delta\omega$ on $X$ for some $\delta>0$. Also, since $[\alpha]$ is nef, for every $\ve>0$ there is a smooth function $\rho_\ve$ such that
$\alpha+\ddbar\rho_\ve\geq -\ve\omega$ on $X$. Now since $[E]|_{E}\cong\mathcal{O}_{\mathbb{CP}^{n-1}}(-1)$, there is a smooth closed real $(1,1)$ form $\eta$ on $\ti{X}$ which is cohomologous to $[E]$, is supported on a neighborhood $U$ of $E$, and such that $\mu^*\omega-\gamma\eta$ is a K\"ahler form on $\ti{X}$ for some small $\gamma>0$ (see e.g. \cite[Lemma 3.5]{DP}).
For every $\ve>0$ choose a large constant $C_\ve>0$ such that
$$\ti{\rho}_\ve:=\widetilde{\max}(\rho_\ve-C_\ve,\psi),$$
is a smooth function on $X$ which agrees with $\psi$ in a neighborhood of $x$ which contains $\mu(U)$. We have $\alpha+\ddbar\ti{\rho}_\ve\geq -\ve\omega$ on $X$ and $\alpha+\ddbar\ti{\rho}_\ve\geq \delta\omega$ on $\mu(U)$. Therefore on $U$ we have
$$\mu^*\alpha-\delta\gamma\eta+\ddbar(\mu^*\ti{\rho}_\ve)>0,$$
while on $\ti{X}\backslash U$ we have
$$\mu^*\alpha-\delta\gamma\eta+\ddbar(\mu^*\ti{\rho}_\ve)=\mu^*\alpha+\ddbar(\mu^*\ti{\rho}_\ve)\geq-\ve\mu^*\omega\geq -C\ve\ti{\omega},$$
where $\ti{\omega}$ is a fixed K\"ahler form on $\ti{X}$, with $\mu^*\omega\leq C\ti{\omega}$.
This shows that $\mu^*[\alpha]-\delta\gamma[E]$ is nef, and so $\ve(\alpha,x)\geq \gamma\delta>0$. We obtain that \eqref{zero} holds, and therefore also \eqref{uno} in the case when $x\in \Null(\alpha)$.

It remains to show that \eqref{uno} holds. First, we assume that $[\alpha]$ is K\"ahler, which implies that $\ve(\alpha,x)>0$.
By definition we have that $\mu^*[\alpha]-\ve(\alpha,x)[E]$ is nef but not K\"ahler (because the K\"ahler cone is open), so by Theorem \ref{dp} there is a positive-dimensional irreducible analytic subvariety $\ti{V}\subset \ti{X}$ with
\begin{equation}\label{zzero}
0=\int_{\ti{V}}(\mu^*\alpha-\ve(\alpha,x)[E])^{k},
\end{equation}
where $k=\dim \ti{V}$.
If $\ti{V}$ is disjoint from $E$ then $V=\mu(\ti{V})$ would be an irreducible $k$-dimensional subvariety of $X$ and we would have
$$0=\int_{\ti{V}}(\mu^*\alpha)^k=\int_V\alpha^k,$$
a contradiction since $[\alpha]$ is K\"ahler.

If $\ti{V}\subset E$ then $-[E]|_{\ti{V}}\cong \mathcal{O}_{\mathbb{CP}^{n-1}}(1)|_{\ti{V}}$ which is ample, and so we would have
$$0=\ve(\alpha,x)^k\int_{\ti{V}}(-[E])^{k}>0,$$
a contradiction.

Therefore we must have that $\ti{V}$ intersects $E$ but is not contained in it, and so $V=\mu(V)$ is an irreducible $k$-dimensional subvariety of $X$ which passes through $x$ and \eqref{zzero} gives
$$\ve(\alpha,x)=\left(\frac{\int_{V}\alpha^{k}}{\mathrm{mult}_x V}\right)^{\frac{1}{k}},$$
thus showing \eqref{uno}, with the infimum being in fact a minimum.

Lastly, if $[\alpha]$ is just nef then we have
\[\begin{split}
\ve(\alpha,x)&=\lim_{\ve\downarrow 0}\ve(\alpha+\ve\omega,x)=\lim_{\ve\downarrow 0}\min_{V\ni x}\left(\frac{\int_{V}(\alpha+\ve\omega)^{\dim V}}{\mathrm{mult}_x V}\right)^{\frac{1}{\dim V}}\\
&\geq \inf_{V\ni x}\left(\frac{\int_{V}\alpha^{\dim V}}{\mathrm{mult}_x V}\right)^{\frac{1}{\dim V}},\end{split}\]
which combined with \eqref{part} proves \eqref{uno}.
\end{proof}

\subsection{The K\"ahler-Ricci flow}

We briefly discuss an application of Theorem \ref{ct} to the study of the K\"ahler-Ricci flow (see for example \cite{To2} for a detailed exposition). Let $(X,\omega_0)$ be a compact K\"ahler manifold, and let $\omega(t)$ be a family of K\"ahler forms on $X$ which solve the K\"ahler-Ricci flow equation
\begin{equation}\label{krf}
\frac{\de}{\de t}\omega(t)=-\Ric(\omega(t)),\quad \omega(0)=\omega_0,
\end{equation}
for $t\in [0,T),$ where $0<T<\infty$ is the maximal existence time, which we assume is finite. Here $\Ric(\omega(t))$ is the Ricci curvature form of $\omega(t)$ which equals $R_{h(t)}$ where $h(t)$ is the metric on $K_X^*$ induced by $\det g(t)$. It is known that the maximal existence time $T$ is finite if and only if $c_1(K_X)$ is not nef \cite{TZ,Ts}. At time $T$ a finite-time singularity forms, and the metrics $\omega(t)$ cannot converge everywhere smoothly to a limiting K\"ahler metric on $X$. The cohomology classes $[\omega(t)]$ however do converge to the limiting class $[\alpha]=[\omega_0]+Tc_1(K_X)$, which is nef but not K\"ahler.

Define the singularity formation set $\Sigma$ of this flow to be equal to the complement of
$$\{x\in X\ |\ \exists U\ni x \textrm{ open, }\exists \omega_T \textrm{ K\"ahler metric on }U \textrm{ s.t. } \omega(t)\overset{C^\infty(U)}{\to}\omega_T \textrm{ as }t\to T^{-}\}.$$
A conjecture of Feldman-Ilmanen-Knopf \cite{FIK} states that $\Sigma$ should be an analytic subvariety of $X$. This was proved in \cite{CT}, and more precisely we have
\begin{thm}[Collins-T. \cite{CT}]\label{krr}
For any finite-time singularity of the K\"ahler-Ricci flow we have that
$$\Sigma=\Null(\alpha).$$
\end{thm}
In other words, $\Sigma$ equals the union of all analytic subvarieties whose volume shrinks to zero as $t$ approaches $T$.
In particular $\Sigma=X$ happens if and only if $\int_X\alpha^n=0$. In this case, we expect that $X$ admits a Fano fibration onto a lower-dimensional normal compact K\"ahler space $Y$, and this was proved recently by Zhang and the author \cite{ToZ} when $n\leq 3$.

Theorem \ref{ct} is used crucially in the proof of Theorem \ref{krr} to produce suitable barrier functions (with analytic singularities along $\Null(\alpha)$), which are used to prove that the metrics $\omega(t)$ have a smooth limit on every compact set in $X\backslash\Null(\alpha)$.

The case when the maximal existence time $T$ is infinite is quite different, see e.g. \cite{TWY, TZ2} and references therein.

\subsection{Degenerations of Calabi-Yau metrics}
Let now $X$ be a compact K\"ahler manifold with $c_1(K_X)=0$ in $H^{1,1}(X,\mathbb{R})$ (equivalently, $K_X$ is torsion in $\mathrm{Pic}(X)$). We will call such a manifold Calabi-Yau. By Yau's Theorem \cite{Y}, every K\"ahler class on $X$ contains a unique representative which is a Ricci-flat K\"ahler metric, i.e. a K\"ahler form $\omega$ with $\Ric(\omega)\equiv 0$.

Let now $[\alpha]$ be nef and big, and let $[\alpha_t]$ be a path of K\"ahler classes ($0<t\leq 1$) which converge to $[\alpha]$ as $t\to 0$. By Yau's Theorem, for every $t>0$ there is a unique Ricci-flat K\"ahler form $\omega_t$ in the class $[\alpha_t]$, and the question we would like to address is what is the behavior of these metrics as $t\to 0$.

\begin{thm}\label{cy}
In this setup there is an incomplete Ricci-flat K\"ahler metric $\omega_0$ on $X\backslash\Null(\alpha)$ (which depends only on $[\alpha]$ and not on the path $[\alpha_t]$), such that
$$\omega_t\to\omega_0 \quad \textrm{ as }\quad t\to 0,$$
smoothly on compact subsets of $X\backslash\Null(\alpha)$. Furthermore, $(X,\omega_t)$ converge in the Gromov-Hausdorff topology (i.e. as metric spaces) to the metric completion of $(X\backslash\Null(\alpha), \omega_0).$
\end{thm}

Again Theorem \ref{ct} is used to construct barrier functions and prove estimates on the metrics $\omega_t$. Theorem \ref{cy} was proved in \cite{CT} building upon earlier work in \cite{To3,RZ,EGZ,BEGZ}. When $[\alpha]$ is a rational class, it follows that $X$ is projective and the base-point-free theorem gives a birational morphism $f:X\to Y$ onto a singular Calabi-Yau variety with at worst canonical singularities, with $\mathrm{Exc}(f)=\Null(\alpha)$, and then $\omega_0$ can be thought of as the pullback of a singular Ricci-flat metric on $Y$ as constructed in \cite{EGZ}. In this case it follows from \cite{So} (see also \cite{DS}) that the Gromov-Hausdorff limit as above is in fact homeomorphic to $Y$.

The case when $[\alpha]$ is nef but not big has also been widely studied, and in this case (assuming that $[\alpha]$ is the pullback of a K\"ahler class from the base of a fiber space) the Ricci-flat metrics collapse in the limit to a lower-dimensional space, see \cite{GW,To4, GTZ,GTZ2,HT,TWY}.

\section{Ideas from the proof}\label{sectproof}

In this section we describe the proof of Theorem \ref{ct}. The easy part is showing that
$$\Null(\alpha)\subset E_{nK}(\alpha).$$
This was observed already in \cite[Theorem 2.2]{CZ}, which we roughly follow here.
If this was not the case, we could find a point $x\in \Null(\alpha)$ which is not in $E_{nK}(\alpha)$, and so there is a K\"ahler current $T=\alpha+\ddbar\psi\geq \delta\omega$ on $X$, ($\delta>0$) with analytic singularities, and which is smooth near $x$. Let $V$ be any positive-dimensional irreducible analytic subvariety of $X$ which passes through $x$, and let $k=\dim V$. As explained in subsection \ref{nakas}, since $X$ is in class $\mathcal{C}$, so is $V$, and so we can take a resolution $\mu:\ti{V}\to V$ obtained as a composition of blowups with smooth centers, where $\ti{V}$ is a compact K\"ahler manifold. Furthermore, regarding the first center of blowups as a submanifold of $X$, blowing it up inside $X$ and repeating, we obtain an extension of the map $\mu$ to $\mu:\ti{X}\to X$ where $\ti{X}$ is a compact complex manifold which contains $\ti{V}$ as a submanifold.
From this it is clear that $\mu^*\alpha|_{\ti{V}}$ defines a smooth closed real $(1,1)$ form on $\ti{V}$, whose class $[\mu^*\alpha|_{\ti{V}}]$ is nef, and so
for every $\ve>0$ we can find a smooth function $\vp_\ve$ on $\ti{V}$ such that $\mu^*\alpha|_{\ti{V}}+\ddbar\vp_\ve\geq -\ve\ti{\omega}$ on $\ti{V}$, where $\ti{\omega}$ is a K\"ahler form on $\ti{V}$. Also, the current $\mu^*T$ can be restricted to $\ti{V}$, and exactly as in \eqref{kc} we obtain a quasi-psh function $\psi'$ on $\ti{X}$ (not identically $-\infty$ on $\ti{V}$) such that $T':=\mu^*T+\ddbar\psi'$ is a K\"ahler current on $\ti{X}$, and it satisfies $T'|_{\ti{V}}\geq\delta'\ti{\omega}$ for some $\delta'>0$.
Then we have
\[\begin{split}\int_{\ti{V}}(\mu^*\alpha+\ve\ti{\omega})^k&=\int_{\ti{V}}(\mu^*\alpha+\ve\ti{\omega})\wedge(\mu^*\alpha+\ve\ti{\omega}+\ddbar\vp_\ve)^{k-1}\\
&=\int_{\ti{V}}(T'+\ve\ti{\omega}-\ddbar(\mu^*\psi+\psi'))\wedge(\mu^*\alpha+\ve\ti{\omega}+\ddbar\vp_\ve)^{k-1}\\
&=\int_{\ti{V}}(T'+\ve\ti{\omega})\wedge(\mu^*\alpha+\ve\ti{\omega}+\ddbar\vp_\ve)^{k-1}\\
&\geq \delta'\int_{\ti{V}}\ti{\omega}\wedge(\mu^*\alpha+\ve\ti{\omega}+\ddbar\vp_\ve)^{k-1},
\end{split}\]
using that
$$\int_{\ti{V}}\ddbar(\mu^*\psi+\psi')\wedge\gamma=0,$$
for any closed smooth real $(n-1,n-1)$ form $\gamma$. Iterating this argument, we obtain
$$\int_{\ti{V}}(\mu^*\alpha+\ve\ti{\omega})^k\geq \delta'^{k}\int_{\ti{V}}\ti{\omega}^k,$$
and letting $\ve\to 0$ we finally get
$$\int_V\alpha^k=\int_{\ti{V}}(\mu^*\alpha)^k=\lim_{\ve\to 0}\int_{\ti{V}}(\mu^*\alpha+\ve\ti{\omega})^k\geq \delta'^{k}\int_{\ti{V}}\ti{\omega}^k>0,$$
and since $V$ was arbitrary this shows that $x\not\in \Null(\alpha)$, a contradiction.

The reverse inclusion
$$E_{nK}(\alpha)\subset\Null(\alpha),$$
is much harder to prove, and we will give an outline of the argument, referring to \cite{CT} for full details.

We again argue by contradiction, so suppose we had a point $x\in E_{nK}(\alpha)$ which is not in $\Null(\alpha)$. Let $V$ be an irreducible component of $E_{nK}(\alpha)$ that passes through $x$. Thanks to Proposition \ref{isola}, $V$ is a positive-dimensional irreducible analytic subvariety, say $\dim V=k$. Using an embedded resolution of singularities, it is not hard to see \cite[p.1180]{CT} that we can assume without loss of generality that $V$ is smooth, so it is a compact complex manifold, in class $\mathcal{C}$. The class $[\alpha|_V]$ is nef, and since $x\not\in\Null(\alpha)$ it satisfies
$$\int_V\alpha^{k}>0.$$
Therefore $[\alpha|_V]$ is big by Theorem \ref{dpmass}, and so it contains a K\"ahler current $T=\alpha|_V+\ddbar\vp\geq \delta \omega|_V$, $\delta>0$, where we may assume that the function $\vp$ on $V$ has analytic singularities defined by a coherent ideal sheaf $\mathcal{I}$ on $V$. By Theorem \ref{bouck} we also have a K\"ahler current $K=\alpha+\ddbar\psi$ on $X$, with analytic singularities along $E_{nK}(\alpha)$, so in particular $\psi|_{V}\equiv-\infty$.

The goal now is to use $T$, together with $K$, to produce a global K\"ahler current $\ti{T}=\alpha+\ddbar\ti{\vp}$ on $X$ in the class $[\alpha]$ such that $\ti{\vp}|_{V}$ is smooth on a Zariski open subset of $V$. This is exactly the content of \cite[Theorem 3.2]{CT}. Once we achieve this, it follows that applying Demailly's regularization \cite{Dem} to $\ti{T}$ produces a K\"ahler current with analytic singularities which is smooth at the generic point of $V$, contradicting the fact that $V$ is a component of $E_{nK}(\alpha).$

For the sake of clarity, let us first see how to construct $\ti{T}$ in the case when $n=2$ and $E_{nK}(\alpha)=V$. In this case $V$ is a compact Riemann surface, and $[\alpha|_V]$ is nef and big, which in fact implies that $[\alpha|_V]$ is a K\"ahler class. Indeed $E_{nK}(\alpha|_V)$ is an analytic subvariety of $V$, hence a finite set of points, but we know from Proposition \ref{isola} that there cannot be any such points. We then choose $\omega$ a K\"ahler form on $V$ in the class $[\alpha|_V]$, which will now play the role of the K\"ahler current $T$ above. Since $V$ is a smooth submanifold of $X$ it is elementary to find an extension of $\omega$, still denoted by $\omega$, to a K\"ahler form on a neighborhood $U$ of $V$ in $X$, which is of the form $\omega=\alpha+\ddbar\rho$ on $U$, for a smooth function $\rho$. Since the global function $\psi$ equals $-\infty$ on $V$ and is smooth outside of $V$, there is a large constant $C$ such that $\rho-C<\psi$ in a neighborhood of $\de U$. We can then set
\begin{equation}\label{glue}
   \ti{\vp} = \left\{
     \begin{array}{ll}
     \max(\rho-C,\psi), &  \text{ on }U\\
     \psi &  \text{ on } X\backslash U,
     \end{array}
   \right.
\end{equation}
and we have that $\ti{\vp}$ is now globally defined, satisfies that $\ti{T}=\alpha+\ddbar\ti{\vp}$ is a K\"ahler current on $X$, and $\ti{\vp}|_V$ equals $\rho|_V$ which is smooth. This completes the proof when $n=2$.

It should be now clear what the difficulties are in extending this argument when $n>2$. The K\"ahler current $T$ that we have produced on $V$ will in general have singularities, and it is not clear anymore how to produce an extension to a neighborhood $U$ of $V$ in $X$. Furthermore, even if we could produce this extension, it would still have singularities and so the simple gluing in \eqref{glue} would not work to produce a global K\"ahler current.

To deal with the extension problem, the first observation is that in fact it is enough to achieve an extension on some bimeromorphic model of $X$ (which roughly speaking corresponds to finding an extension to a ``pinched neighborhood'' of $V$ inside $X$, cf. the discussion in \cite{CT2}). This is because if $\mu:\ti{X}\to X$ is a sequence of smooth blowups, which is an isomorphism at the generic point of $V$, then if we can achieve our extension and gluing goal on $\ti{X}$, we can then simply push forward the resulting K\"ahler current from $\ti{X}$ to $X$.

The advantage of working on a blowup is that, by resolving the ideal sheaf $\mathcal{I}$ defining the singularities of the K\"ahler current $T$ on $V$, we obtain a modification $\mu:\ti{X}\to X$ as above, such that $\mu^*T$ has analytic singularities described by $E\cap \ti{V}$, where $\ti{V}$ is the proper transform of $V$ and $E$ is an effective $\mathbb{R}$-divisor on $\ti{X}$ whose support has simple normal crossings, and also has normal crossings with $\ti{V}$. Explicitly, this means that there is an open cover $\{W_j\}_{1\leq j\leq N}$ of $\ti{V}$ by coordinate charts for $\ti{X}$, such that on each $W_j$ there are coordinates $(z_1,\dots,z_n)$ such that
$\ti{V}\cap W_j=\{z_1=\cdots=z_{n-k}=0\}$, and with $\mathrm{Supp}(E)\cap W_j=\{z_{i_1}\cdots z_{i_p}=0\},$ where $n-k<i_1,\dots,i_p\leq n$, and on $\ti{V}\cap W_j$ we have
\begin{equation}\label{sing}
\mu^*\vp=c\log\left(\prod_k |z_{i_k}|^{2\alpha_{i_k}}\right)+h_j,
\end{equation}
where $c,\alpha_{i_k}\in\mathbb{R}_{>0},$ and $h_j$ is a continuous function on $W_j$. As in \eqref{kc}, after adding a small correction term to $\mu^*\vp$, which is singular only along $E\cap\ti{V}$, we obtain a function $\vp'$ with analytic singularities so that $\mu^*\alpha+\ddbar\vp'$ defines a K\"ahler current on $\ti{V}$. If on $W_j$ we write $z=(z_1,\dots,z_{n-k}), z'=(z_{n-k+1},\dots,z_n)$, then we can extend $\vp'$ to a function $\vp_j$ on $W_j$ by setting
$$\vp_j(z,z')=\vp'(0,z')+A|z|^2,$$
for $A>0$ large, so that $\mu^*\alpha+\ddbar\vp_j$ is a K\"ahler current on $W_j$ (with analytic singularities).

Next, we would like to patch together these functions $\vp_j$ to a function defined on $U:=\bigcup_j W_j$, which is an open neighborhood of $\ti{V}$. If the functions $\vp_j$ were continuous, this would follow immediately from a well-known gluing procedure of Richberg \cite{Ri}. The key point is that on each nontrivial overlap $W_j\cap W_k$ the difference $\vp_j-\vp_k$ is in fact continuous, because the analytic singularities of both $\vp_j$ and $\vp_k$ are of exactly the same type, along the divisor $E\cap W_j\cap W_k$, as in \eqref{sing}. This is all that is needed for the Richberg gluing argument to go through, and we thus obtain a function $\rho$ on $U$, with analytic singularities described by $E\cap U$, such that $\mu^*\alpha+\ddbar\rho$ is a K\"ahler current on $U$.

Now that we have achieved our desired extension, albeit on some blowup of $X$, we proceed to the gluing step. The pullback $\mu^*K=\mu^*\alpha+\ddbar(\mu^*\psi)$ can also be modified by adding a small correction term to it (which is singular only along $E$) so that we obtain a global K\"ahler current $\mu^*\alpha+\ddbar\psi'\geq\delta'\ti{\omega}$ on $\ti{X}$, with analytic singularities along $\mu^{-1}(E_{nK}(\alpha))\cup \mathrm{Exc}(\mu)\supset\ti{V}$. In order for $\psi'$ to glue to $\rho$, we first add a small correction term to $\rho$, which is singular along the closure of $\mu^{-1}(E_{nK}(\alpha))\backslash\ti{V}$, to obtain a K\"ahler current $\mu^*\alpha+\ddbar\rho'$ on $U$ with analytic singularities wherever $\psi'$ is singular, except along $\ti{V}$.

Since the class $[\mu^*\alpha]$ is nef, for every $\ve>0$ there is a smooth function $\vp_\ve$ on $\ti{X}$ such that $\mu^*\alpha+\ddbar\vp_\ve\geq -\ve\delta'\ti{\omega}$, so that
$$\mu^*\alpha+\ddbar\left(\ve\psi'+(1-\ve)\vp_\ve\right)\geq \ve^2\delta'\ti{\omega}>0,$$
is still a K\"ahler current, whose singularities have been attenuated. Choosing $\ve>0$ sufficiently small, it is not hard to show that there are a neighborhood $\ti{U}$ of $\ti{V}$ in $\ti{X}$, with $\ti{U}\subset U$, and a constant $C>0$ such that
$$\ve\psi'+(1-\ve)\vp_\ve>\rho'-C,$$
on a neighborhood of $\de\ti{U}$.  We can then set
\begin{equation}\label{glue2}
   \ti{\vp} = \left\{
     \begin{array}{ll}
     \max(\rho'-C,\ve\psi'+(1-\ve)\vp_\ve), &  \text{ on }\ti{U}\\
     \ve\psi'+(1-\ve)\vp_\ve &  \text{ on } \ti{X}\backslash \ti{U},
     \end{array}
   \right.
\end{equation}
and we have that $\ti{\vp}$ is now globally defined, satisfies that $\ti{T}=\mu^*\alpha+\ddbar\ti{\vp}$ is a K\"ahler current on $\ti{X}$, and $\ti{\vp}|_{\ti{V}}$ equals $\rho'|_{\ti{V}}-C$ which is smooth at the generic point of $\ti{V}$. This completes the proof of Theorem \ref{ct}.\\

We conclude this article by noting that a refinement of the extension and gluing techniques that we just presented allowed Collins and the author \cite{CT2} to prove the following extension theorem for K\"ahler currents:

\begin{thm}
Let $(X,\omega)$ be a compact K\"ahler manifold and $V\subset X$ a submanifold. If $T$ is a K\"ahler current on $V$ with analytic singularities in the class $[\omega|_V]$ then $T$ extends to a global K\"ahler current on $X$ in the class $[\omega]$.
\end{thm}

It is expected that this extension result should hold more generally when $V$ is an analytic subvariety and $T$ is just a closed positive current in the class $[\omega|_V]$ (in which case $T$ should extend to a global closed positive current). This was proved by Coman-Guedj-Zeriahi \cite{CGZ} when $X$ is projective and $[\omega]$ is a rational class, using rather different techniques.

\section{The problem of effectivity}\label{secteff}

In this section we briefly discuss the problem of effectivity in Nakamaye's theorem. This issue was already raised by Nakamaye in \cite[p.553]{Na2}, and more recently reiterated in \cite[p.105]{Na3}, in view of applications to Diophantine geometry. It was also mentioned explicitly in \cite{Den2} in connections with hyperbolicity problems.

Let $X$ be a projective manifold with an ample line bundle $A$ and a nef and big line bundle $L$. By Nakamaye's Theorem \ref{naka} there is an $\ve\in\mathbb{Q}_{>0}$ such that
\begin{equation}\label{eqq}
\mathbb{B}(L-\ve A)=\Null(L),
\end{equation}
where as before $\mathbb{B}$ denotes the stable base locus.
The supremum of all such $\ve$ will be denoted by $\mu(L,A)>0$. The problem of effectivity is to give explicit/effective lower bounds for $\mu(L,A)$, for example depending only on intersection numbers of $L$ and $A$.

We can easily translate this question into analytic language. If $X$ is now a compact K\"ahler manifold with a K\"ahler form $\omega$ and a big and nef $(1,1)$ class $[\alpha]$, we define
$\mu'([\alpha],[\omega])$ to be the supremum of all $\ve>0$ such that there exists a K\"ahler current $T$ on $X$ in the class $[\alpha]$ with analytic singularities, with $T\geq \ve\omega$ on $X$ and with $\mathrm{Sing}(T)=\Null(\alpha)$. Note that, as the notation suggests, this depends only on the class of $\omega$. Thanks to Theorem \ref{ct} together with Boucksom's result in \eqref{equal}, we know that $\mu'([\alpha],[\omega])>0$.

The following proposition is essentially the same as Theorem \ref{bplus}:
\begin{prop}
Let $X$ be a projective manifold with an ample line bundle $A$ and a nef and big line bundle $L$. Then we have
$$\mu(L,A)=\mu'(c_1(L),c_1(A)).$$
\end{prop}
\begin{proof}
Let $0<\ve<\mu(L,A)$, and choose a Hermitian metric $h_A$ on $A$ whose curvature is a K\"ahler form $\omega\in c_1(A)$ and a smooth Hermitian metric $h$ on $L$, with curvature form $\alpha\in c_1(L)$. Choose $m\geq 1$ such that $m\ve\in\mathbb{N}$ and $\mathrm{Bs}(m(L-\ve A))=\mathbb{B}(L-\ve A)=\Null(L)$. If $\{s_1,\dots,s_N\}$ is a basis of $H^0(X,m(L-\ve A))$ then
$$T=\alpha+\frac{\sqrt{-1}}{2\pi m}\de\dbar\log\sum_i |s_i|^2_{h^m\otimes h_A^{-m\ve}}\geq\ve\omega,$$
is a K\"ahler current on $X$ in the class $c_1(L)$ with analytic singularities along $\mathbb{B}(L-\ve A)=\Null(L)$, and so
$\mu'(c_1(L),c_1(A))\geq \ve$, and since $\ve<\mu(L,A)$ is arbitrary we obtain $\mu(L,A)\leq\mu'(c_1(L),c_1(A))$.

On the other hand, given $0<\ve<\mu'(c_1(L),c_1(A))$, we can find a K\"ahler current $T$ on $X$ in the class $c_1(L)$ which has analytic singularities along $\Null(L)$, and with $T\geq \ve\omega$. Given any $x\in X\backslash \Null(L)$, choose a coordinate patch $U$ containing $x$ so that $T$ is smooth on $U$, and let $\theta$ be a smooth cutoff function supported in $U$ and identically $1$ near $x$. Then
$$\ti{T}=T-\ve\omega+\ve'\ddbar(\theta\log|z-x|^2),$$
is a K\"ahler current for $\ve'>0$ sufficiently small, which is smooth on $U\backslash\{x\}$ and with Lelong number $\ve'$ at $x$. Following the exact same argument as in the proof of Theorem \ref{bplus}, using H\"ormander's $L^2$ estimates for $\overline{\partial}$, we obtain a global section of $m(L-\ve A),$ for $m\geq 1$ sufficiently large, which is nonvanishing at $x$. We conclude that $x\not\in\mathbb{B}(L-\ve A)$, and so we have shown that \eqref{eqq} holds. Hence
$\mu(L,A)\geq \ve$, and since $\ve<\mu'(c_1(L),c_1(A))$ is arbitrary we obtain $\mu(L,A)\geq\mu'(c_1(L),c_1(A))$.
\end{proof}

Now that we have recast the effectivity problem in Nakamaye's Theorem in analytic terms, it is easy to appreciate its difficulty. Indeed, on the one hand Boucksom's argument in \eqref{equal} shows that there is a K\"ahler current $T$ in the class $[\alpha]$ with analytic singularities along $\Null(\alpha)$ but does not quantify its positivity. On the other hand, the Demailly-P\u{a}un mass concentration technique \cite{DP} produces another K\"ahler current $\ti{T}$ in the class $[\alpha]$ with analytic singularities and with
$\ti{T}\geq\ve\omega,$ for any given
$$\ve<\frac{\int_X\alpha^n}{n\int_X\alpha^{n-1}\wedge\omega},$$
see \cite[Theorem 2.3]{To} for this result. This is a very attractive numerical bound, but unfortunately there is no control over the singularities of $\ti{T}$, which in general will be much larger than $\Null(\alpha)$. The tension between these two competing requirements, large positivity and small singularities, makes it extremely hard to give nontrivial lower bounds for $\mu(L,A)$ in general.


\begin{thebibliography}{99}
\bibitem{Sesh} T. Bauer, S. Di Rocco, B. Harbourne, M. Kapustka, A. Knutsen, W. Syzdek, T. Szemberg {\em A primer on Seshadri constants}, in {\em Interactions of classical and numerical algebraic geometry}, 33--70, Contemp. Math., 496, Amer. Math. Soc., Providence, RI, 2009.
\bibitem{Bi} C. Birkar {\em The augmented base locus of real divisors over arbitrary fields}, arXiv:1312.0239.
\bibitem{BoT} S. Boucksom \emph{C\^{o}nes positifs des vari\'et\'es complexes compactes}, Ph.D. Thesis, Institut Fourier Grenoble, 2002.
\bibitem{Bo} S. Boucksom {\em On the volume of a line bundle}, Internat. J. Math. {\bf 13} (2002), no. 10, 1043--1063.
\bibitem{Bo2} S. Boucksom {\em Divisorial Zariski decompositions on compact complex manifolds}, Ann. Sci. \'Ecole Norm. Sup. (4) {\bf 37} (2004), no. 1, 45--76.
\bibitem{BBP} S. Boucksom, A. Broustet, G. Pacienza {\em Uniruledness of stable base loci of adjoint linear systems via Mori Theory}, Math. Z. {\bf 275} (2013), no. 1-2, 499--507.
\bibitem{BCL} S. Boucksom, S. Cacciola, A.F. Lopez {\em Augmented base loci and restricted volumes on normal varieties}, Math. Z. {\bf 278} (2014), no. 3-4, 979--985.
\bibitem{BEGZ} S. Boucksom, P. Eyssidieux, V. Guedj, A. Zeriahi \emph{ Monge-Amp\`ere equations in big cohomology classes}, Acta Math. {\bf 205} (2010), no. 2, 199--262.
\bibitem{CL} S. Cacciola, A.F. Lopez {\em Nakamaye's theorem on log canonical pairs}, Ann. Inst. Fourier (Grenoble) {\bf 64} (2014), no. 6, 2283--2298.
\bibitem{CHP} F. Campana, A. H\"oring, T. Peternell {\em Abundance for K\"ahler threefolds}, to appear in Ann. Sci. \'Ecole Norm. Sup.
\bibitem{CZ} S. Cantat, A. Zeghib {\em Holomorphic actions, Kummer examples, and Zimmer program},  Ann. Sci. \'Ecole Norm. Sup. (4) {\bf 45} (2012), no. 3, 447--489.
\bibitem{CH} J. Cao, A. H\"oring {\em Rational curves on compact K\"ahler manifolds}, arXiv:1502.03936.
\bibitem{CMM} P. Cascini, J. McKernan, M. Musta\c{t}\u{a} {\em The augmented base locus in positive characteristic},  Proc. Edinb. Math. Soc. (2) {\bf 57} (2014), no. 1, 79--87.
\bibitem{Ch} I. Chiose {\em The K\"ahler rank of compact complex manifolds}, J. Geom. Anal. {\bf 26} (2016), no. 1, 603--615.
\bibitem{CT2} T.C. Collins, V. Tosatti {\em An extension theorem for K\"ahler currents with analytic singularities}, Ann. Fac. Sci. Toulouse Math. {\bf 23} (2014), no. 4, 893--905.
\bibitem{CT} T.C. Collins, V. Tosatti {\em K\"ahler currents and null loci}, Invent. Math. {\bf 202} (2015), no.3, 1167--1198.
\bibitem{CT3} T.C. Collins, V. Tosatti {\em A singular Demailly-P\u{a}un theorem}, C. R. Math. Acad. Sci. Paris {\bf 354} (2016), no. 1, 91--95.
\bibitem{CGZ} D. Coman, V. Guedj, A. Zeriahi {\em Extension of plurisubharmonic functions with growth control}, J. Reine Angew. Math. {\bf 676} (2013), 33--49.
\bibitem{Dem87} J.-P. Demailly {\em Une preuve simple de la conjecture de Grauert-Riemenschneider}, in {\em S\'eminaire d'Analyse P. Lelong-P. Dolbeault-H. Skoda, Ann\'ees 1985/1986}, 24--47, Lecture Notes in Math., 1295, Springer, Berlin, 1987.
\bibitem{Dem} J.-P. Demailly \emph{Regularization of closed positive currents and intersection theory}, J. Algebraic Geom. {\bf 1} (1992), no. 3, 361--409.
\bibitem{Dem92} J.-P. Demailly \emph{Singular Hermitian metrics on positive line bundles}, in {\em Complex algebraic varieties (Bayreuth, 1990),} 87--104, Lecture Notes in Math., 1507, Springer, Berlin, 1992.
\bibitem{Dem2} J.-P. Demailly \emph{A numerical criterion for very ample line bundles}, J. Differential Geom. {\bf 37} (1993), no. 2, 323--374.
\bibitem{Demb} J.-P. Demailly {\em Complex Analytic and Differential Geometry}, available on the author's webpage.
\bibitem{DP} J.-P. Demailly, M. P\u{a}un \emph{Numerical characterization of the K\"ahler cone of a compact K\"ahler manifold}, Ann. of Math., {\bf 159} (2004), no. 3, 1247--1274.
\bibitem{Den} Y. Deng {\em Transcendental Morse inequality and generalized Okounkov bodies}, arXiv:1503.00112.
\bibitem{Den2} Y. Deng {\em Effectivity in hyperbolicity-related problems}, arXiv:1606.03831.
\bibitem{DS} S.K. Donaldson, S. Sun {\em Gromov-Hausdorff limits of K\"ahler manifolds and algebraic geometry}, Acta Math. {\bf 213} (2014), no. 1, 63--106.
\bibitem{Ein} L. Ein {\em Linear systems with removable base loci}, Comm. Algebra {\bf 28} (2000), no. 12, 5931--5934.
\bibitem{ELMNP2} L. Ein, R. Lazarsfeld, M. Musta\c{t}\u{a}, M. Nakamaye, M. Popa {\em Asymptotic invariants of base loci}, Ann. Inst. Fourier (Grenoble) {\bf 56} (2006), no. 6, 1701--1734.
\bibitem{ELMNP} L. Ein, R. Lazarsfeld, M. Musta\c{t}\u{a}, M. Nakamaye, M. Popa {\em Restricted volumes and base loci of linear series}, Amer. J. Math. {\bf 131} (2009), no. 3, 607--651.
\bibitem{EGZ} P. Eyssidieux, V. Guedj, A. Zeriahi {\em Singular K\"ahler-Einstein metrics}, J. Amer. Math. Soc. {\bf 22} (2009), 607--639.
\bibitem{FIK} M. Feldman, T. Ilmanen, D. Knopf {\em Rotationally symmetric shrinking and expanding gradient K\"ahler-Ricci solitons},  J. Differential Geom. {\bf 65}  (2003),  no. 2, 169--209.
\bibitem{Fu2} A. Fujiki {\em Closedness of the Douady spaces of compact K\"ahler spaces}, Publ. Res. Inst. Math. Sci. {\bf 14} (1978), no. 1, 1--52.
\bibitem{Fu} T. Fujita {\em Semipositive line bundles}, J. Fac. Sci. Univ. Tokyo Sect. IA Math. {\bf 30} (1983), no. 2, 353--378.
\bibitem{GH} P. Griffiths, J. Harris {\em Principles of algebraic geometry}, Pure and Applied Mathematics. Wiley-Interscience, New York, 1978.
\bibitem{GTZ} M. Gross, V. Tosatti, Y. Zhang {\em Collapsing of abelian fibred Calabi-Yau manifolds}, Duke Math. J. {\bf 162} (2013), no. 3, 517--551.
\bibitem{GTZ2} M. Gross, V. Tosatti, Y. Zhang {\em Gromov-Hausdorff collapsing of Calabi-Yau manifolds}, Comm. Anal. Geom. {\bf 24} (2016), no.1, 93--113.
\bibitem{GW} M. Gross, P.M.H. Wilson \emph{Large complex structure limits of $K3$ surfaces}, J. Differential Geom. {\bf 55} (2000), no. 3, 475--546.
\bibitem{HT} H.-J. Hein, V. Tosatti {\em Remarks on the collapsing of torus fibered Calabi-Yau manifolds}, Bull. Lond. Math. Soc. {\bf 47} (2015), no.6, 1021--1027.
\bibitem{HP} A. H\"oring, T. Peternell {\em Minimal models for K\"ahler threefolds}, Invent. Math. {\bf 203} (2016), no. 1, 217--264.
\bibitem{Ke} S. Keel {\em Basepoint freeness for nef and big line bundles in positive characteristic}, Ann. of Math. (2) {\bf 149} (1999), no. 1, 253--286.
\bibitem{Kl} S.L. Kleiman {\em Toward a numerical theory of ampleness}, Ann. of Math. (2) {\bf 84} (1966), 293--344.
\bibitem{Laz} R. Lazarsfeld {\em Positivity in algebraic geometry I \& II}, Springer-Verlag, Berlin, 2004.
\bibitem{LWX} C. Li, X. Wang, C. Xu {\em Quasi-projectivity of the moduli space of smooth K\"ahler-Einstein Fano manifolds}, arXiv:1502.06532.
\bibitem{Na} M. Nakamaye \emph{Stable base loci of linear series}, Math. Ann. {\bf 318} (2000), no. 4, 837--847.
\bibitem{Na2} M. Nakamaye {\em Base loci of linear series are numerically determined}, Trans. Amer. Math. Soc. {\bf 355} (2003), no. 2, 551--566.
\bibitem{Na3} M. Nakamaye {\em Roth's theorem: an introduction to diophantine approximation}, in {\em Rational Points, Rational Curves, and Entire Holomorphic Curves on Projective Varieties},  75--108, Contemp. Math., 654, Amer. Math. Soc., Providence, RI, 2015.
\bibitem{Pa} M. P\u{a}un {\em Sur l'effectivit\'e num\'erique des images inverses de fibr\'es en droites}, Math. Ann. {\bf 310} (1998), no. 3, 411--421.
\bibitem{Ri} R. Richberg {\em Stetige streng pseudokonvexe Funktionen}, Math. Ann. {\bf 175} (1968), 257--286.
\bibitem{RZ} X. Rong, Y. Zhang {\em Continuity of extremal transitions and flops for Calabi-Yau manifolds}, J.  Differential Geom. {\bf 89} (2011), no. 2, 233--269.
\bibitem{ST} G. Schumacher, H. Tsuji {\em Quasi-projectivity of moduli spaces of polarized varieties}, Ann. of Math., {\bf 159} (2004), no. 2, 597--639.
\bibitem{Si2} Y.-T. Siu {\em Some recent results in complex manifold theory related to vanishing theorems for the semipositive case}, in {\em Workshop Bonn 1984 (Bonn, 1984)}, 169--192, Lecture Notes in Math., 1111, Springer, Berlin, 1985.
\bibitem{So} J. Song {\em Riemannian geometry of K\"ahler-Einstein currents}, arXiv:1404.0445.
\bibitem{Ta} S. Takayama {\em A local ampleness criterion of torsion free sheaves}, Bull. Sci. Math. {\bf 137} (2013), no. 5, 659--670.
\bibitem{TZ} G. Tian, Z. Zhang {\em On the K\"ahler-Ricci flow on projective manifolds of general type}, Chinese Ann. Math. Ser. B {\bf 27} (2006), no. 2, 179--192.
\bibitem{To3} V. Tosatti \emph{Limits of Calabi-Yau metrics when the K\"ahler class degenerates}, J. Eur. Math. Soc. (JEMS) {\bf 11} (2009), no.4, 755-776.
\bibitem{To4} V. Tosatti {\em Adiabatic limits of Ricci-flat K\"ahler metrics}, J. Differential Geom. {\bf 84} (2010), no.2, 427--453.
\bibitem{To} V. Tosatti {\em The Calabi-Yau Theorem and K\"ahler currents}, Adv. Theor. Math. Phys. {\bf 20} (2016), no.2, 381--404.
\bibitem{To2} V. Tosatti {\em KAWA lecture notes on the K\"ahler-Ricci flow}, to appear in Ann. Fac. Sci. Toulouse Math.
\bibitem{TWY} V. Tosatti, B. Weinkove, X. Yang {\em The K\"ahler-Ricci flow, Ricci-flat metrics and collapsing limits}, arXiv:1408.0161.
\bibitem{TZ2} V. Tosatti, Y. Zhang {\em Infinite time singularities of the K\"ahler-Ricci flow}, Geom. Topol. {\bf 19} (2015), no.5, 2925--2948.
\bibitem{ToZ} V. Tosatti, Y. Zhang {\em Finite time collapsing of the K\"ahler-Ricci flow on threefolds}, arXiv:1507.08397.
\bibitem{Ts} H. Tsuji {\em Degenerate Monge-Amp\`ere equation in algebraic geometry.}, in {\em Miniconference on Analysis and Applications (Brisbane, 1993)}, 209--224, Proc. Centre Math. Appl. Austral. Nat. Univ., 33, Austral. Nat. Univ., Canberra, 1994.
\bibitem{Y} S.-T. Yau  {\em On the Ricci curvature of a compact K\"ahler manifold and the complex Monge-Amp\`ere equation, I}, Comm. Pure Appl. Math. {\bf 31} (1978), no.3, 339--411.
\bibitem{Za} O. Zariski {\em The theorem of Riemann-Roch for high multiples of an effective divisor on an algebraic surface}, Ann. of Math. (2) {\bf 76} (1962), 560--615.
\end{thebibliography}
\end{document}